\documentclass[11pt]{amsart}
\usepackage{amsmath,amssymb,etoolbox,mathrsfs,booktabs,float,enumerate,calc}

\theoremstyle{plain}
\newtheorem{theorem}{Theorem}

\newtheorem*{proposition*}{Proposition}
\newtheorem{lemma}[theorem]{Lemma}
\newtheorem{lemmma*}{Lemma}
\newtheorem{corollary}[theorem]{Corollary}
\theoremstyle{definition}

\newtheorem*{definition*}{Definition}

\newtheorem{remark*}{Remark}

\newcommand{\Irr}{\mathrm{Irr}}

\newcommand{\pchar}{\theta}
\newcommand{\const}{c}
\newcommand{\bZ}{\mathbf Z}

\renewcommand{\S}{S}
\newcommand{\A}{A}
\numberwithin{equation}{section}
\numberwithin{theorem}{section}
\numberwithin{table}{section}
\begin{document}
\title[Covering numbers for characters]{Covering numbers for characters of symmetric groups}
\author[A.\ R.\ Miller]{Alexander Rossi Miller}
\address{Center for Communications Research, Princeton}
\begin{abstract}
If $n>4$ and $c(\theta)$ denotes the set of irreducible 
constituents of a character $\theta$, then $c(\chi^k)=\Irr(S_n)$ 
for all nonlinear $\chi\in\Irr(S_n)$ if and only if $k\geq n-1$.
\end{abstract}
\maketitle
\thispagestyle{empty}
\section{Introduction}
Let $\chi$ denote a character of a finite group $G$.
Let $c(\chi)$ denote the set of irreducible constituents of $\chi$.

Burnside \cite{Bu} proved that if $\chi$ is faithful,
then the constituents of the various powers of $\chi$ cover the set of irreducible characters of $G$ in the sense~that 
\begin{equation}\label{b cover}
  c(\chi)\cup c(\chi^2)\cup \ldots =\Irr(G).
\end{equation}
When \eqref{b cover} holds, we denote by $d(\chi)$ the smallest positive integer such that
\begin{equation}
  c(\chi)\cup c(\chi^2)\cup\ldots \cup c(\chi^{d(\chi)})=\Irr(G).
\end{equation}
It is not always true that $\Irr(G)=c(\chi^N)$ for some $N$ when $\chi$ is faithful, as can be seen by taking any faithful irreducible character of a nontrivial cyclic group. 
Steinberg and Brauer gave independent
proofs of Burnside's theorem \cite{Br,St}. Brauer's proof supplies  
the bound $d(\chi)\leq |\{\chi(g): g \in G\}|$,
but little else is known about these numbers in general.

If $\chi$ has trivial center, 
then by Lemma~\ref{lemma e exists}, all irreducible characters of $G$ appear in a single power of~$\chi$ in the
sense that, for some $N$,
\begin{equation}\label{p cover}
  c(\chi^N)=\Irr(G).
\end{equation}
When \eqref{p cover} holds, we denote by $e(\chi)$ the 
smallest positive integer such that
\begin{equation}
  c(\chi^{e(\chi)})=\Irr(G).
\end{equation}
For example, if $G$ is the dihedral group with $14$ elements,
then each nonlinear $\chi\in\Irr(G)$ has trivial center and satisfies ${e(\chi)=6}$ and $d(\chi)=3$. 

All irreducible characters $\chi$ of $\S_n$ with $\chi(1)>1$ have trivial center when $n>4$,
so for each of these characters of $\S_n$ the numbers
$d(\chi)$ and $e(\chi)$ exist.
The $e(\chi)$'s for $\S_n$ have been studied asymptotically and for very special
choices of $\chi$, but remain unknown in almost all cases.
In this paper, we consider the $e(\chi)$'s (resp.\ $d(\chi)$'s)
all together and find the explicit maximum.
These two maximum values are the smallest positive integers $N$ and $M$
such that every nonlinear  $\chi\in\Irr(\S_n)$ satisfies 
$c(\chi^N)=\Irr(\S_n)$ and $c(\chi)\cup c(\chi^2)\cup\ldots\cup c(\chi^M)=\Irr(\S_n)$. We call these \emph{covering numbers}.

\begin{theorem}\label{main theorem}
  Let $n>4$. Let $X=\{\chi\in\Irr(\S_n) : \chi(1)>1\}$. 
  \begin{enumerate}[\rm (i)\ ]
    \item $c(\chi^k)=\Irr(\S_n)$ for all $\chi\in X$ if and only if $k\geq n-1$.
    \item $c(\chi)\cup c(\chi^2)\cup\ldots \cup c(\chi^k)=\Irr(\S_n)$ for all $\chi\in X$ if and only if $k\geq n-1$.
    \end{enumerate}
  \end{theorem}

\section{Supplemental notation, basic lemmas, and remarks}
\begin{lemma}\label{c multiply}
  Let $\chi,\varphi$ be characters of a finite group $G$.
  \begin{enumerate}[\rm (i)]
  \item If $c(\chi)=c(\varphi)$, then 
    $c(\psi\chi)=c(\psi\varphi)$
    for every character $\psi$ of $G$.
  \item If $c(\chi)=\Irr(G)$, then
  $c(\psi\chi)=\Irr(G)$
  for every character $\psi$ of $G$.
  \end{enumerate}
\end{lemma}

\begin{proof}
  (i) 
  $c(\psi\chi)=\bigcup_{\alpha,\beta} c(\alpha\beta)$ where $\alpha\in c(\psi)$ and $\beta\in c(\chi)$.
  (ii) 
  Let $\rho$ be the character of the regular representation. If $c(\chi)=\Irr(G)$, then by the first part we have
  $c(\psi\chi)=c(\psi\rho)=c(\psi(1)\rho)=\Irr(G)$.
\end{proof}

\begin{corollary}\label{stab cor}
  If $c(\chi^n)=\Irr(G)$, then $c(\chi^k)=\Irr(G)$ for all $k\geq n$.
\end{corollary}
\begin{proof}
  This follows from Lemma~\ref{c multiply}.
\end{proof}

For a character $\chi$ of a finite group $G$, let $\bZ(\chi)=\{g\in G: |\chi(g)|=\chi(1)\}$.

\begin{lemma}\label{lemma e exists}
  If $\chi$ is a character of a finite group $G$ and 
  $|\bZ(\chi)|=1$, then $c(\chi^n)=\Irr(G)$ for some $n$.
\end{lemma}

\begin{proof}
  If $\chi$ satisfies $|\bZ(\chi)|=1$, then for any
  $\theta\in\Irr(G)$, the inner product $\langle \chi^k,\theta\rangle$
  is asymptotic to $\chi(1)^k\theta(1)/|G|$ as $k\to\infty$, since
  \[\langle \chi^k,\theta\rangle\frac{|G|}{\chi(1)^k\theta(1)}=1+\sum_{g\neq 1}\frac{\chi(g)^k}{\chi(1)^k}\frac{\overline{\theta(g)}}{\theta(1)}\to 1 \text{ as }k\to\infty,\]
  so $\langle \chi^k,\theta\rangle$ tends to infinity as $k\to\infty$.
\end{proof}

\begin{lemma}\label{triv center}
  Let $\chi$ be a character of a finite group $G$.
  \begin{enumerate}[\rm (a)]
  \item If $|\bZ(G)|=1$, then $|\bZ(\chi)|=1$ if and only if $\chi$ is faithful.
  \item If $\chi\in\Irr(G)$, then $|\bZ(\chi)|=1$ if and only if $|\bZ(G)|=1$ and $\chi$ is faithful.
  \end{enumerate}
\end{lemma}

\begin{proof}
  (a) Since $\bZ(\chi)/\ker \chi\subseteq \bZ(G/\ker \chi)$, we have ${|\bZ(\chi)|=1}$ if $|\bZ(G)|=|\ker \chi|=1$.
 The inclusion $\ker\chi\subseteq \bZ(\chi)$ gives the other direction. Part (b) follows from (a) since $\bZ(G)=\bigcap_{\theta\in\Irr(G)}\bZ(\theta)$.
\end{proof}

\begin{lemma}\label{triv center nonlinear}
  Let $n>4$ and let $\chi$ be an irreducible character of $S_n$. Then
  $|\bZ(\chi)|=1$ if and only if $\chi(1)>1$. 
\end{lemma}

\begin{proof}
  If $|\bZ(\chi)|=1$, then $\chi(1)>1$ because linear characters have full~center.
  Suppose 
  $|\bZ(\chi)|\neq 1$. Then by Lemma~\ref{triv center}, $\chi$ is not faithful. 
  Since the only nontrivial normal subgroups of $\S_n$ are $\A_n$
  and $\S_n$ when $n>4$, we therefore have 
  $\chi_{\A_n}=k1_{\A_n}$ for some~$k$. But 
  the restriction of any irreducible of
  $\S_n$ to $\A_n$ is either irreducible
  or the sum of two distinct irreducibles, so $k=1$.
\end{proof}

\begin{corollary}\label{exist}
  If $n>4$, then for each $\chi\in\Irr(\S_n)$ with $\chi(1)>1$, there exists
  a smallest positive integer $e(\chi)$ such that $c(\chi^{e(\chi)})=\Irr(\S_n)$.
\end{corollary}
\begin{proof}
  This follows from Lemma~\ref{lemma e exists} and Lemma~\ref{triv center nonlinear}.
\end{proof}

Let  ${\mathcal N(G)=\{\chi\in\Irr(G): \chi(1)>1\}}$ for any finite group $G$.
In view of Corollary~\ref{exist}, we shall write, for $n>4$, 
\[
  e(\S_n)=\max\{e(\chi): \chi\in\mathcal N(\S_n)\},\quad
  d(\S_n)=\max\{d(\chi): \chi\in\mathcal N(\S_n)\}.
\]
By Corollary~\ref{stab cor}, for each $k\geq e(\S_n)$, $j\geq d(\S_n)$, and $\chi\in\mathcal N(\S_n)$, 
\[c(\chi^k)=\Irr(\S_n),\quad c(\chi)\cup c(\chi^2)\cup\ldots\cup c(\chi^j)=\Irr(\S_n).\]

We remark that the analogue of Corollary~\ref{exist} holds also for many other interesting groups. For example, finite simple groups.
For each finite nonabelian simple group $G$, it would be interesting to
find the analogue of Theorem~\ref{main theorem}, i.e.\ to determine the values
\[
  e(G)=\max\{e(\chi): \chi\in\mathcal N(G)\},\quad
  d(G)=\max\{d(\chi): \chi\in\mathcal N(G)\}.
\]

\section{Proof of Theorem 1.1}
By a partition $\lambda$ of $n$, we mean an infinite weakly decreasing sequence of non-negative integers
$\lambda=(\lambda_1,\lambda_2,\ldots)$ such that $\sum \lambda_i=n$.
The number of nonzero terms in $\lambda$ is denoted $\ell(\lambda)$ and
we often abbreviate $\lambda$ to $(\lambda_1,\lambda_2,\ldots, \lambda_{\ell(\lambda)})$.
The conjugate of $\lambda$ is the partition $\lambda'=(\lambda'_1,\lambda'_2,\ldots)$ with 
 $\lambda'_i=|\{j: \lambda_j\geq i\}|$. 
We also identify $\lambda$ with its diagram of boxes, so $|\lambda|=n$ and if $\mu$ is another partition, then $|\lambda\smallsetminus \mu|$ is the number of boxes in $\lambda$ that lie outside of $\mu$. Finally, $\chi_\lambda$ denotes the irreducible character of $\S_n$ corresponding to $\lambda$ in the usual way.

We denote by $+$ and $+'$ the following two ways of adding partitions:
\begin{equation}
  \lambda+\mu=(\lambda_1+\mu_1,\lambda_2+\mu_2,\ldots),
  \qquad
  \lambda+'\mu=(\lambda'+\mu')'.
\end{equation}
With this notation, we have the following important lemma first observed by Klyachko in \cite{K} and proved  by Christandl--Harrow--Mitchison in \cite{CHM}. The second part of Lemma~\ref{semigroup} involving
$+'$ is a restatement of the first part because $\chi_\lambda\chi_\mu=\chi_{\lambda'}\chi_{\mu'}$, and this restatement also appears in \cite{S}, where it is used to study asymptotics. We use Lemma~\ref{semigroup} multiple times throughout this section. 

\begin{lemma}\label{semigroup} 
  If $\chi_\nu\in \const(\chi_\lambda\chi_\mu)$
  and
  ${\chi_\gamma\in \const(\chi_\alpha\chi_\beta)}$,
  then
  \begin{equation}
    \chi_{\nu+\gamma}\in \const(\chi_{\lambda+\alpha}\chi_{\mu+\beta})
    \quad\text{and}\quad
    \chi_{\nu+\gamma}\in \const(\chi_{\lambda+'\alpha}\chi_{\mu+'\beta}).
  \end{equation}
\end{lemma}

For integers $0\leq u\leq n$, we denote by $\theta_{n,u}$ the induced character 
\begin{equation}
  \theta_{n,u}=(1_{\S_{n-u}})^{\S_n}.
\end{equation}

\begin{lemma}\label{theta move}
  For integers $0\leq u\leq n$ and any $\chi_\lambda\in\Irr(\S_n)$,
  \begin{equation}
    \const(\pchar_{n,u}^v\chi_{\lambda})
    =
    \{\chi_\mu :   |\lambda\smallsetminus \mu|\leq uv \},\quad v\in\mathbb N.
  \end{equation}
\end{lemma}

\begin{proof}
  For any class function $\chi$ of $\S_n$, we have 
  $\pchar_{n,u}\chi=(\chi_{\S_{n-u}})^{\S_n}$.
  From this and the usual combinatorial rule for restricting and inducing characters between $\S_{n-1}$ and $\S_n$, we obtain the $v=1$ case of the claim.
  Repeatedly applying this special case gives the general case.
\end{proof}

We say that a partition $\lambda$ is a rectangle or is rectangular if \[\lambda_1=\lambda_2=\ldots=\lambda_{\ell(\lambda)}.\]

\begin{lemma}\label{non-rectangle}
  Let $\lambda$ be a non-rectangular partition of $n$. Then
  ${c(\pchar_{n,2})\subseteq c(\chi_\lambda^2)}$.
\end{lemma}

\begin{proof}
  The case $n=3$ holds by inspection, so assume $n\geq 4$. 
   Then 
  \[c(\pchar_{n,2})=\{\chi_{(n)},\chi_{(n-1,1)},\chi_{(n-2,2)},\chi_{(n-2,1,1)}\}.\]
  Now, $\chi_{(n)}$ belongs to $c(\chi_\lambda^2)$
  because
  $\langle \chi_\lambda^2,\chi_{(n)}\rangle=1$.
  For $\chi_{(n-1,1)}$, by using Frobenius reciprocity, we have
  \[
    \langle\chi_\lambda^2,\chi_{(n-1,1)}\rangle
    =
    \langle \chi_\lambda^2,\pchar_{n,1}-1\rangle
    =
    \langle (\chi_\lambda)_{S_{n-1}},(\chi_\lambda)_{S_{n-1}}\rangle-1
    =
    |\{i: \lambda_i>\lambda_{i+1}\}|-1.
  \]
  Since $\lambda$ is not a rectangle, the last quantity is positive.
  
  For $\chi_{(n-2,1,1)}$, we first write $\lambda$ as
  \[\lambda=\alpha+'\gamma\]
  with $\alpha=(\lambda_2,\lambda_3,\ldots,\lambda_{\ell(\lambda)-1})$
  and $\gamma=(\lambda_1,\lambda_{\ell(\lambda)})$. 
  Since $\lambda$ is not a rectangle, 

  \[\gamma=\beta+(2,1)\]
  for some partition $\beta$. 
  Since $\chi_{(|\beta|)}\in c(\chi_\beta^2)$ and $\chi_{(1,1,1)}\in c(\chi_{(2,1)}^2)$,
  we have by Lemma~\ref{semigroup},
  \begin{equation}\label{bog}
    \chi_{(|\beta|)+(1,1,1)}\in c(\chi_\gamma^2).
  \end{equation}
  From \eqref{bog} and the fact that $\chi_{(|\alpha|)}\in c(\chi_\alpha^2)$, another application of Lemma~\ref{semigroup} yields
  \[
    \chi_{(n-2,1,1)}
    =
    \chi_{(|\alpha|)+(|\beta|)+(1,1,1)}
    \in
    c(\chi_{\alpha+'\gamma}^2)
    =
    c(\chi_{\lambda}^2).
  \]

  For $\chi_{(n-2,2)}$, we consider separately the cases
  $\lambda_2=1$ and $\lambda_2\geq 2$.
  If $\lambda_2\geq 2$, then 
  \[\lambda=(\alpha+(2,2))+'\beta\]
  for some partitions $\alpha$ and $\beta$. 
  In this case, since $\chi_{(|\alpha|)}\in c(\chi_\alpha^2)$,
  ${\chi_{(2,2)}\in c(\chi_{(2,2)}^2)}$,
  and $\chi_{(|\beta|)}\in c(\chi_\beta^2)$,
  we conclude that
  \[
    \chi_{(n-2,2)}
    =
    \chi_{(|\alpha|)+(2,2)+(|\beta|)}
    \in
    c(\chi_{(\alpha+(2,2))+'\beta}^2)
    =
    c(\chi_\lambda^2).
  \]
  If $\lambda_2=1$, let $\alpha\subseteq \lambda$ be a partition of $4$ with at least two nonzero parts, so that $\alpha$ is either $(2,1,1)$ or $(3,1)$, and 
  \[\lambda=(\alpha+\beta)+'\gamma\]
  for some partitions $\beta$ and $\gamma$. 
  Since $\chi_{(2,2)}\in c(\chi_\alpha^2)$, $\chi_{(|\beta|)}\in c(\chi_\beta^2)$, and
  $\chi_{(|\gamma|)}\in c(\chi_\gamma^2)$, we then have 
  \[\chi_{(n-2,2)}=\chi_{(2,2)+(|\beta|)+(|\gamma|)}\in c(\chi_\lambda^2).\]
  This concludes the proof.
\end{proof}

\begin{lemma}\label{rectangle}
  Let $\lambda$ be a rectangular partition of $n>6$ such that $\chi_\lambda(1)\neq 1$. Then
  ${c(\pchar_{n,5})\subseteq c(\chi_\lambda^4)}$. 
\end{lemma}

\begin{proof}
  The lemma holds by inspection if $n<10$, so assume $n\geq 10$.
  
  Let $\mu=(n-4,2,2)$. We claim that
  \begin{equation}\label{first claim}
    \chi_\mu\in c(\chi_\lambda^2).
  \end{equation}
  Since $\chi_\lambda^2=\chi_{\lambda'}^2$,
  we may assume that
  $\lambda_1\geq \ell(\lambda)$.
  To establish \eqref{first claim},
  first suppose $\ell(\lambda)=2$.
  Then for some partition $\alpha$,
  \[
    \lambda=(4,4)+\alpha.
  \]
  Since
  $\chi_{(4,2,2)}\in c(\chi_{(4,4)}^2)$
  and
  $\chi_{(|\alpha|)}\in c(\chi_\alpha^2)$, 
  \[
    \chi_\mu
    =
    \chi_{(4,2,2)+(|\alpha|)}
    \in
    c(\chi_{(4,4)+\alpha}^2)
    =
    c(\chi_{\lambda}^2).
  \]
  If instead $\ell(\lambda)\geq 3$,
  then for some partitions $\alpha$ and $\beta$, 
  \[
    \lambda=((3,3,3)+\alpha)+'\beta.
  \]
  Since $\chi_{(5,2,2)}\in c(\chi_{(3,3,3)}^2)$,
  \[
    \chi_\mu
    =
    \chi_{(5,2,2)+(|\alpha|)+(|\beta|)}
    \in
    c(\chi_{((3,3,3)+\alpha)+'\beta}^2)
    =
    c(\chi_{\lambda}^2).
  \]
  This establishes the claim in \eqref{first claim}.

  We next claim that
  \begin{equation}\label{second claim}
    c(\pchar_{n,5})\subseteq c(\chi_\mu^2).
  \end{equation}
  Let
  \[\Lambda=\{\alpha+'(10-|\alpha|) :  |\alpha|\leq 5\}.\]
  Then
  \begin{equation}\label{Lambda set}
    c(\pchar_{n,5})=\{\chi_{\nu+(n-10)} : \nu\in\Lambda\}.
  \end{equation}
  By inspection (see Table~\ref{table}, for example), we find
  that $\chi_\nu\in c(\chi_{(6,2,2)}^2)$ for each $\nu\in \Lambda$.
  So by Lemma~\ref{semigroup}, for each $\nu\in\Lambda$, we have
  \[\chi_{\nu+(n-10)}\in c(\chi_{(6,2,2)+(n-10)}^2)=c(\chi_\mu^2).\]
  Therefore, by \eqref{Lambda set},
  \[c(\pchar_{n,5})\subseteq c(\chi_\mu^2).\]

  Combining \eqref{first claim} and \eqref{second claim} gives
  \[
    c(\theta_{n,5})
    \subseteq
    c(\chi_\mu^2)
    \subseteq
    c(\chi_\lambda^4).
  \]
  This concludes the proof.
  \end{proof}

\begin{center}
  \begin{table}[H]\footnotesize
    \caption{\label{table}}
\begin{tabular}{lcclc}
  \toprule
  $\nu\in \Lambda$ & $\langle \chi_\nu,\chi_{(6,2,2)}^2\rangle$ &&
    $\nu\in\Lambda$ & $\langle \chi_\nu,\chi_{(6,2,2)}^2\rangle$ \\
  \cmidrule{1-2}\cmidrule{4-5}
  $(10)$      & $1$ &  &$(6,2,1,1)$     & $7$  \\
  $(9,1)$     & $1$ &  &$(6,1,1,1,1)$   & $4$  \\
  $(8,2)$     & $3$ &  &$(5,5)$         & $1$  \\
  $(8,1,1)$   & $1$ &  &$(5,4,1)$       & $7$  \\
  $(7,3)$     & $3$ &  &$(5,3,2)$       & $8$  \\
  $(7,2,1)$   & $5$ &  &$(5,3,1,1)$     & $10$ \\
  $(7,1,1,1)$ & $3$ &  &$(5,2,2,1)$     & $8$  \\
  $(6,4)$     & $4$ &  &$(5,2,1,1,1)$   & $7$  \\
  $(6,3,1)$   & $7$ &  &$(5,1,1,1,1,1)$ & $2$  \\
  $(6,2,2)$   & $7$ &  &                &      \\
  \bottomrule
\end{tabular}
\end{table}
\end{center}

\begin{proof}[{\bf Proof of Theorem 1.1}]
  Let $\epsilon=(n-1,1)$. Then $\chi_\epsilon=\theta_{n,1}-1$. 
  So by Lemma~\ref{c multiply}, for any $k\geq 1$, we have 
  $c(\chi_\epsilon^{k})\subseteq c(\theta_{n,1}^{k-1}\chi_\epsilon)$.
  By Lemma~\ref{theta move}, if $1\leq k\leq n-2$, then
  $c(\theta_{n,1}^{k-1}\chi_\epsilon)$ does not
  contain $\chi_{(1,1,\ldots,1)}$. So 
  \[n-1\leq d(\S_n)\leq e(\S_n).\]
  It therefore remains to show that $e(\chi_\lambda)\leq n-1$ for all
  $\chi_\lambda\in\Irr(S_n)$ with $\chi_\lambda(1)>1$. We consider separately the case where $\lambda$ is a rectangle and the case where $\lambda$ is not a rectangle.

  Suppose that $\lambda$ is a partition of $n$ that is not a rectangle.
  By Lemma~\ref{non-rectangle},
  \begin{equation}\label{nr-rep}
    c(\theta_{n,2})\subseteq c(\chi_\lambda^2).
  \end{equation}
  Suppose first that $n$ is odd and let $m=\frac{n-1}{2}$. Then by \eqref{nr-rep},
  \begin{equation}\label{theta n2 subset}
    c(\theta_{n,2}^m)\subseteq c(\chi_\lambda^{2m})=c(\chi_\lambda^{n-1}).
  \end{equation}
  By Lemma~\ref{theta move},
  \begin{equation}\label{theta n2 cover}
    c(\theta_{n,2}^m)=c(\theta_{n,2}^m\chi_{(n)})=\Irr(\S_n).
  \end{equation}
  Therefore, by \eqref{theta n2 subset} and \eqref{theta n2 cover}, 
  \[c(\chi_\lambda^{n-1})=\Irr(\S_n).\]
  Now suppose that $n$ is even and let $m=\frac{n-2}{2}$. Then by \eqref{nr-rep},
  \begin{equation}\label{theta n even m}
    c(\theta_{n,2}^m\chi_\lambda)\subseteq c(\chi_\lambda^{n-1}).
  \end{equation}
   By Lemma~\ref{theta move},
   \[c(\theta_{n,2}^m)=c(\theta_{n,2}^m\chi_{(n)})=\Irr(\S_n)\smallsetminus \{\zeta\},\]
   where $\zeta=\chi_{(1,1,\ldots,1)}$ is the sign character.
   Denoting by $\rho$ the character of the regular representation of $\S_n$, it follows that
   \[c(\theta_{n,2}^m)=c(\rho-\zeta).\]
   Therefore, by Lemma~\ref{c multiply},
   \begin{equation}
     c(\theta_{n,2}^m\chi_\lambda)=c((\rho-\zeta)\chi_\lambda).
   \end{equation}
   Since $\chi_\lambda(1)>1$, each $\varphi\in\Irr(\S_n)$ satisfies
   \begin{equation}\label{all positive}
     \langle (\rho-\zeta)\chi_\lambda,\varphi\rangle
     =
     \varphi(1)\chi_\lambda(1)-\langle \chi_{\lambda'},\varphi\rangle>0.
   \end{equation}
   Hence, by \eqref{theta n even m}--\eqref{all positive},
   we have in this case again 
   \[c(\chi_\lambda^{n-1})=\Irr(\S_n).\]
   This concludes the case where $\lambda$ is not a rectangle.

   Suppose now that $\lambda$ is a rectangular partition of $n$ such that $\chi_\lambda(1)>1$. 
   If ${n\leq 12}$, then  
   $c(\chi_\lambda^{n-1})=\Irr(\S_n)$ by inspection.
   Assume $n>12$. 
   Let $m=\lfloor \frac{n-1}{4}\rfloor$. Let $k=4m$.
   Then by Lemma~\ref{rectangle},
   \[c(\theta_{n,5}^m)\subseteq c(\chi_\lambda^k).\]
   By Lemma~\ref{theta move} and the fact that $5m\geq n-1$, we have
   \[c(\theta_{n,5}^m)=\Irr(\S_n).\]
   Therefore, $c(\chi_\lambda^k)=\Irr(\S_n)$. Since $k\leq n-1$,
   we have by Lemma~\ref{c multiply},
   \[c(\chi_\lambda^{n-1})=\Irr(\S_n).\]
   This concludes the proof of Theorem~\ref{main theorem}.
 \end{proof}

 \subsection*{Acknowledgments}
It is a pleasure to thank Gunter Malle and TU Kaiserslautern for their hospitality.
 \medskip

 \noindent
 Travel facilitated in part by DFG grant 286237555.

\end{document}